\documentclass[12pt]{amsart}
\usepackage{amsmath,amssymb,amsfonts,amsthm,amscd,indentfirst}
\usepackage{amsmath,amssymb,amsfonts,amsthm,amscd}
\usepackage{indentfirst}
\usepackage{hyperref}

\numberwithin{equation}{section}

\newtheorem{prop}{Proposition}[section]
\newtheorem{theo}[prop]{Theorem}
\newtheorem{lemm}[prop]{Lemma}

\newtheorem{rema}[prop]{Remark}

\def\and{\quad{\rm and}\quad}

\def\<{\langle}
\def\>{\rangle}

\usepackage{graphicx}

\begin{document}

\title[Interior $C^2$ estimate for Hessian quotient equation]{Interior $C^2$ estimate for Hessian quotient equation in general dimension}
\author[Siyuan Lu]{Siyuan Lu}
\address{Department of Mathematics and Statistics, McMaster University, 
1280 Main Street West, Hamilton, ON, L8S 4K1, Canada.}
\email{siyuan.lu@mcmaster.ca}
\thanks{Research of the author was supported in part by NSERC Discovery Grant.}

\begin{abstract}
In this paper, we study the interior $C^2$ regularity problem for the Hessian quotient equation $\left(\frac{\sigma_n}{\sigma_k}\right)(D^2u)=f$. We give a complete answer to this longstanding problem: for $k=n-1,n-2$, we establish an interior $C^2$ estimate; for $k\leq n-3$, we show that interior $C^2$ estimate fails by finding a singular solution. 
\end{abstract}

\maketitle

\section{Introduction}

In this paper, we study the interior $C^2$ regularity problem for the Hessian quotient equation $\left(\frac{\sigma_n}{\sigma_k}\right)(D^2u)=f$, where $\sigma_k$ is the $k$-th elementary symmetric function. We give a complete answer to this longstanding problem according to the range of $k$.

\medskip

For $k=n-1,n-2$, we establish an interior $C^2$ estimate.
\begin{theo}\label{Theorem}
Let $k=n-1,n-2$ and let $n\geq 3$. Let $f\in C^{1,1}(B_{10}\times\mathbb{R})$ be a positive function and let $u\in C^4(B_{10})$ be a convex solution of 
\begin{align}\label{eq-orginal}
F(D^2u)=\left(\frac{\sigma_n}{\sigma_k}\right)(D^2u)=f(x,u),\quad \textit{in}\quad B_{10}\subset \mathbb{R}^n.
\end{align}
Then we have
\begin{align*}
|D^2u(0)|\leq C,
\end{align*}
where $C$ depends only on $n$, $\|u\|_{C^{0,1}(\overline{B}_9)}$, $\min_{\overline{B}_9\times [-M,M]}f $ and $\|f\|_{C^{1,1}\left( \overline{B}_9\times [-M,M]\right)}$. Here $M$ is a large constant satisfying $\|u\|_{L^\infty(\overline{B}_9)}\leq M$.
\end{theo}

For $k\leq n-3$, we find a singular solution.
\begin{theo}\label{Singular}
Let $1\leq k\leq n-3$ and let $n\geq 3$. Then there exists a convex viscosity solution $u$ of
\begin{align*}
\left(\frac{\sigma_n}{\sigma_k}\right)(D^2u)=f(x),\quad \textit{in}\quad B_r\subset \mathbb{R}^n,
\end{align*}
for some $r>0$ and some smooth positive function $f$ on $B_r$ such that $u\in C^{0,1}(B_r)$ but $u\notin C^{1,\beta}\left( B_{\frac{r}{2}}\right) $ for any $\beta>1-\frac{2}{n-k}$.
\end{theo}

For completeness, we also include the case for $n=2$.
\begin{theo}\label{n=2}
Let $f\in C^{1,1}(B_{10}\times\mathbb{R})$ be a positive function and let $u\in C^4(B_{10})$ be a convex solution of 
\begin{align*}
\left(\frac{\sigma_2}{\sigma_1}\right)(D^2u)=f(x,u),\quad \textit{in}\quad B_{10}\subset \mathbb{R}^2.
\end{align*}
Then we have
\begin{align*}
|D^2u(0)|\leq C,
\end{align*}
where $C$ depends only on $\|u\|_{C^{0,1}(\overline{B}_9)}$, $\min_{\overline{B}_9\times [-M,M]}f $ and $\|f\|_{C^{1,1}\left( \overline{B}_9\times [-M,M]\right)}$. Here $M$ is a large constant satisfying $\|u\|_{L^\infty(\overline{B}_9)}\leq M$.
\end{theo}

\begin{rema}
Using the same idea by Urbas \cite{Urbas}, we can extend our singular solution to show that interior $C^2$ estimate fails for the Hessian quotient equation $\left(\frac{\sigma_k}{\sigma_l} \right)(D^2u)=f$ for $1\leq l<k\leq n$ with $k-l\geq 3$. Consequently, the only remaining cases for general Hessian quotient equation $\left(\frac{\sigma_k}{\sigma_l} \right)(D^2u)=f$ are $\left(\frac{\sigma_k}{\sigma_{k-1}}\right)(D^2u)=f$ and $\left(\frac{\sigma_k}{\sigma_{k-2}}\right)(D^2u)=f$ for $2\leq k\leq n-1$. For curvature equation $\left(\frac{\sigma_k}{\sigma_{k-1}}\right)(\kappa)=f$, interior $C^2$ estimate was proved by Guan and Zhang \cite{GZ}.
\end{rema}

Priori to our result, very little was known concerning the interior $C^2$ estimate for Hessian quotient equation. The only known cases are: $\left(\frac{\sigma_3}{\sigma_1}\right)(D^2u)=1$ in dimension $3$ and $4$ by Chen, Warren and Yuan \cite{CWY} and Wang and Yuan \cite{WdY14} using special Lagrangian structure of the equation; $\left(\frac{\sigma_3}{\sigma_1}\right)(D^2u)=f$ in dimension $3$ and $4$ by Zhou \cite{Zhou} using twisted special Lagrangian structure of the equation; $\left(\frac{\sigma_3}{\sigma_1}\right)(D^2u)=f$ in dimension $3$ by the author \cite{Lu23} using Legendre transform. 

\medskip

The study of interior $C^2$ estimate for fully nonlinear equations is of great interest due to its simple form: namely the estimate does not depend on any information of the solution on the boundary. It dates back to Heinz's work on Weyl's embedding problem. In \cite{H}, Heinz established an interior $C^2$ estimate for the Monge-Amp\`ere equation $\det(D^2u)=f$ for $n=2$, see also recent proofs by Chen, Han and Ou \cite{CHO} and Liu \cite{Liu}. However, interior $C^2$ estimate fails for the Monge-Amp\`ere equation for $n\geq 3$, due to the counter example by Pogorelov \cite{Pb}. Pogorelov's example was extended by Urbas \cite{Urbas} to show that interior $C^2$ estimate fails for the Hessian equation $\sigma_k(D^2u)=f$ for $k\geq 3$. The remaining case for the Hessian equation is $\sigma_2(D^2u)=f$ and it is a longstanding problem.

For $\sigma_2$ equation, a major breakthrough was made by Warren and Yuan. In \cite{WY09}, they obtained an interior $C^2$ estimate for $\sigma_2(D^2u)=1$ for $n=3$. For general right hand side $f$ in dimension $3$, interior $C^2$ estimate was proved recently by Qiu \cite{Q1,Q2}. For $n=4$, interior $C^2$ estimate for $\sigma_2(D^2u)=1$ was established in a cutting edge paper by Shankar and Yuan \cite{SY-23}. For general $n$, interior $C^2$ estimate for $\sigma_2(D^2u)=1$ was proved by McGonagle, Song and Yuan \cite{MSY} under the condition $D^2u> -c(n)$, by Shankar and Yuan \cite{SY-20} under the condition $D^2u>-K$, by Shankar and Yuan \cite{SY-23} under the condition $D^2u\geq -c(n)\Delta u$ and by Mooney \cite{Mooney} under the condition $D^2u>0$. For general $n$ and general right hand side $f$, interior $C^2$ estimate was established by Guan and Qiu \cite{GQ} under the condition $\sigma_3(D^2u)>-A$.

\medskip

Apart from Hessian equation, another known fully nonlinear equation that admits interior $C^2$ estimate is the special Lagrangian equation $\sum_i\arctan(\lambda_i)=\Theta$ \cite{HL}, where $\lambda_i$'s are the eigenvalues of $D^2u$ and $\Theta$ is called the phase. 

The study of interior $C^2$ estimate for special Lagrangian equation was initiated in a pioneer work by Warren and Yuan \cite{WY09} mentioned above. In \cite{WY09}, they obtained an interior $C^2$ estimate for $\sum_i\arctan(\lambda_i)=\frac{\pi}{2}$ for $n=3$ (This is equivalent to $\sigma_2(D^2u)=1$ for $n=3$). Interior $C^2$ estimate for convex solutions of special Lagrangian equation was proved by Chen, Warren and Yuan \cite{CWY}. Interior $C^2$ estimate for special Lagrangian equation with critical and supercritical phase $|\Theta|\geq \frac{(n-2)\pi}{n}$ was established via works of Wang, Warren and Yuan \cite{WY09,WY10,WdY14}. The notion of critical phase was introduced by Yuan \cite{Y06} and it illustrates the dramatic difference between subcritical phase and supercritical phase. For variable phase functions $\theta(x)$, interior $C^2$ estimate was obtained by Bhattacharya and Shankar \cite{BS1,BS2} for convex solutions, by Bhattacharya \cite{B1} for supercritical phase $|\theta(x)|>\frac{(n-2)\pi}{n}$ and by the author \cite{Lu} for critical and supercritical phase $|\theta(x)|\geq\frac{(n-2)\pi}{n}$. Interior $C^2$ estimate for twisted special Lagrangian equation under certain constraints was studied by Zhou \cite{Zhou}. Note that interior $C^2$ estimate fails for special Lagrangian equation with subcritical phase by examples of Nadirashvili and Vlăduţ \cite{NV} and Wang and Yuan \cite{WdY13}. This illustrates the subtlety of the interior $C^2$ estimate.

\medskip

Compared to Hessian equation and special Lagrangian equation, the structure of Hessian quotient equation is much more complicated. This makes the problem of interior $C^2$ estimate extremely difficult. In our previous work \cite{Lu23}, we successfully attacked the problem in dimension $3$ inspired by the work of Shankar and Yuan \cite{SY-20}. The main strategy is as follows. Let $b=\ln\lambda_1$, where $\lambda_1$ is the largest eigenvalue of $D^2u$. The first step is to establish a Jacobi inequality for $b$. The second step is to bound $b$ via its integral using Legendre transform. The third step is to control the integral using integration by parts.

The above strategy works well in dimension $3$, but we encounter significant difficulties when generalizing to dimension $n$. More specifically, step one and step three. First and foremost, the computation for the Jacobi inequality in \cite{Lu23} is very complicated and the method is unlikely to work in general dimension. More precisely, we need to find a better way to control the third order terms. The key observation here is to establish a concavity inequality for Hessian quotient operator by rewriting the equation in a different form. This new method is very straightforward and much simpler. It allows us to control the third order terms in general dimension. We believe this new method can be used in other problems of fully nonlinear equations. The second difficulty is to control the integral involving $b$. Unlike in dimension $3$, we have many more terms to take care of after each integration by parts. These terms involve different orders of $D^2u$ as well as $\nabla b$. We need to find a systematic way to control these terms. To achieve that, we establish an induction lemma to bound all these terms by two terms. We then bound these two terms separately and complete the proof.

\medskip

The organization of the paper is as follows. In Section 2, we will collect some basic properties of the operator $F$ and a simple observation of our equation. In Section 3, we will prove the crucial concavity inequality for Hessian quotient operator. In Section 4, we will establish a Jacobi inequality for $b$. In Section 5, we will perform a Legendre transform to bound $b$ by its integral. In Section 6, we will complete the proof of Theorem \ref{Theorem} and Theorem \ref{n=2} using integration by parts. In Section 7, we will provide a singular solution for $k\leq n-3$.

\section{Preliminaries}

In this section, we will collect some basic properties of the operator $F$ and a simple observation of our equation.

\medskip

Let $\lambda=(\lambda_1,\cdots,\lambda_n)\in \mathbb{R}^n$, we will denote 
\begin{align*}
(\lambda|i)=(\lambda_1,\cdots,\lambda_{i-1},\lambda_{i+1},\cdots,\lambda_n)\in \mathbb{R}^{n-1},
\end{align*}
i.e. $(\lambda|i)$ is the vector obtained by deleting the $i$-th component of the vector $\lambda$. Similarly, $(\lambda|ij)$ is the vector obtained by deleting the $i$-th and $j$-th components of the vector $\lambda$.

We first collect some basic formulas for Hessian operator, see for instance \cite{LT}.
\begin{lemm}\label{Sigma_k-Lemma-0}
For any $\lambda=(\lambda_1,\cdots,\lambda_n)\in \mathbb{R}^n$, we have
\begin{align*}
\sigma_k(\lambda)= \lambda_i\sigma_{k-1}(\lambda|i)+\sigma_k(\lambda|i),\quad \sum_i\sigma_{k}(\lambda|i)&=(n-k)\sigma_{k}(\lambda),\\
\sum_i\lambda_i\sigma_{k-1}(\lambda|i)= k\sigma_k(\lambda),\quad \sum_i\lambda_i^2\sigma_{k-1}(\lambda|i)= \sigma_1(\lambda)&\sigma_k(\lambda)-(k+1)\sigma_{k+1}(\lambda).
\end{align*}
\end{lemm}

We now collect some basic properties of the operator $F$.
\begin{lemm}\label{F^{ijkl}}
Let $1\leq k<n$ and let $n\geq 2$. Let $u$ be a convex function. Suppose $D^2u$ is diagonalized at $x_0$. Then at $x_0$, we have
\begin{align*}
F^{pq}=\frac{\partial F}{\partial u_{pq}}=\frac{\sigma_n^{pq}}{\sigma_k}-\frac{\sigma_n\sigma_k^{pq}}{\sigma_k^2},
\end{align*}
\begin{align*}
F^{pq,rs}=\frac{\partial^2F}{\partial u_{pq}\partial u_{rs}}= \begin{cases}
-2\frac{\sigma_n^{pp}\sigma_k^{pp}}{\sigma_k^2}+2\frac{\sigma_n\left(\sigma_k^{pp}\right)^2}{\sigma_k^3},\quad &p=q=r=s,\\
\frac{\sigma_n^{pp,rr}}{\sigma_k}-\frac{\sigma_n^{pp}\sigma_k^{rr}}{\sigma_k^2}-\frac{\sigma_n^{rr}\sigma_k^{pp}}{\sigma_k^2}-\frac{\sigma_n\sigma_k^{pp,rr}}{\sigma_k^2}+2\frac{\sigma_n\sigma_k^{pp}\sigma_k^{rr}}{\sigma_k^3}, &p=q,r=s,p\neq r,\\
\frac{\sigma_n^{pq,qp}}{\sigma_k}-\frac{\sigma_n\sigma_k^{pq,qp}}{\sigma_k^2}, & p=s,q=r,p\neq q,\\
0, & \textit{otherwise}.
\end{cases}
\end{align*}

In particular, for all $p\neq q$, we have $-F^{pq,qp}>0$  and
\begin{align*}
-F^{pq,qp}=\frac{F^{pp}-F^{qq}}{\lambda_q-\lambda_p},\quad\lambda_p\neq \lambda_q,
\end{align*}
where $\lambda_p$ and $\lambda_q$ are the eigenvalues of $D^2u$. 
\end{lemm}

\begin{proof}
The first part follows from standard computation. For the second part, we have
\begin{align*}
-F^{pq,qp}=&\ -\frac{\sigma_n^{pq,qp}}{\sigma_k}+\frac{\sigma_n\sigma_k^{pq,qp}}{\sigma_k^2}\\
=&\ \frac{\sigma_n^{pp,qq}}{\sigma_k}-\frac{\sigma_n\sigma_k^{pp,qq}}{\sigma_k^2}\\
=&\ \frac{\sigma_n}{\sigma_k}\left(\frac{1}{\lambda_p\lambda_q}-\frac{\sigma_k^{pp,qq}}{\sigma_k}\right)\\
=&\ \frac{\sigma_n}{\sigma_k}\cdot\frac{\sigma_k-\lambda_p\lambda_q\sigma_k^{pp,qq}}{\lambda_p\lambda_q\sigma_k}>0.
\end{align*}
We have used the fact that $u$ is a convex function in the last inequality.

Moreover, for $\lambda_p\neq \lambda_q$, we have
\begin{align*}
-F^{pq,qp}=&\ \frac{\sigma_n^{pp,qq}}{\sigma_k}-\frac{\sigma_n\sigma_k^{pp,qq}}{\sigma_k^2}= \frac{1}{\sigma_k}\cdot\frac{\sigma_n^{pp}-\sigma_n^{qq}}{\lambda_q-\lambda_p}-\frac{\sigma_n}{\sigma_k^2}\cdot\frac{\sigma_k^{pp}-\sigma_k^{qq}}{\lambda_q-\lambda_p}\\
=&\ \frac{1}{\lambda_q-\lambda_p}\left(\frac{\sigma_n^{pp}}{\sigma_k}-\frac{\sigma_n\sigma_k^{pp}}{\sigma_k^2}\right)-\frac{1}{\lambda_q-\lambda_p}\left(\frac{\sigma_n^{qq}}{\sigma_k}-\frac{\sigma_n\sigma_k^{qq}}{\sigma_k^2}\right)\\
=&\ \frac{F^{pp}-F^{qq}}{\lambda_q-\lambda_p}.
\end{align*}

\end{proof}

We also need the following simple observations of our equation.

\begin{lemm}\label{lambda_3-1}
Let $k=n-1$ and let $n\geq 2$. Let $u$ be a convex solution of (\ref{eq-orginal}). Suppose $D^2u$ is diagonalized at $x_0$ such that $\lambda_1\geq \cdots\geq \lambda_n$. Then at $x_0$, we have
\begin{align*}
f\leq \lambda_n\leq C(n)f,
\end{align*}
where $C(n)$ is a positive constant depending only on $n$.
\end{lemm}

\begin{proof}
We have
\begin{align*}
f=\frac{\sigma_n}{\sigma_{n-1}}\geq \frac{\lambda_1\cdots\lambda_n}{C(n)\lambda_1\cdots\lambda_{n-1}}= \frac{\lambda_n}{C(n)},
\end{align*}
i.e. $\lambda_n\leq C(n)f$.

Similarly,
\begin{align*}
f=\frac{\sigma_n}{\sigma_{n-1}}\leq \frac{\lambda_1\cdots\lambda_n}{\lambda_1\cdots\lambda_{n-1}}=\lambda_n,
\end{align*}
i.e. $\lambda_n\geq f$.

\end{proof}

\begin{lemm}\label{lambda_3}
Let $k=n-2$ and let $n\geq 3$. Let $u$ be a convex solution of (\ref{eq-orginal}). Suppose $D^2u$ is diagonalized at $x_0$ such that $\lambda_1\geq \cdots\geq \lambda_n$. Then at $x_0$, we have
\begin{align*}
\lambda_n\leq \sqrt{C(n)f},\quad \lambda_{n-1}\geq \sqrt{\frac{f}{C(n)}},
\end{align*}
where $C(n)$ is a positive constant depending only on $n$.
\end{lemm}

\begin{proof}
We have
\begin{align*}
f=\frac{\sigma_n}{\sigma_{n-2}}\geq \frac{\lambda_1\cdots\lambda_n}{C(n)\lambda_1\cdots\lambda_{n-2}}\geq \frac{\lambda_n^2}{C(n)},
\end{align*}
i.e. $\lambda_n\leq \sqrt{C(n)f}$.

It follows that
\begin{align*}
\lambda_{n-1}=\frac{f\sigma_{n-2}}{\lambda_1\cdots\lambda_{n-2}\cdot\lambda_n}\geq \frac{f}{\lambda_n}\geq \frac{f}{\sqrt{C(n)f}}=\sqrt{\frac{f}{C(n)}}.
\end{align*}

\end{proof}

\section{Concavity inequality}

In this section, we will prove a concavity inequality for Hessian quotient operator, which is crucial in deriving the Jacobi inequality. The key observation here is to rewrite the Hessian quotient equation in a different form. 

The following two lemmas were pointed out to us by Professor Pengfei Guan. We would like to thank him for his generosity for sharing the lemmas. 

\medskip

For $k=n-1$, the proof is quite straightforward.
\begin{lemm}\label{Con-lemm-1}[Guan and Sroka \cite{GS}]
Let $n\geq 2$ and let $F=\frac{\sigma_n}{\sigma_{n-1}}$ with $\lambda_1\geq\cdots\geq\lambda_n>0$. Then for any $\xi\in \mathbb{R}^n$, we have
\begin{align*}
-\sum_{i,j}F^{ii,jj}\xi_i\xi_j-\frac{F^{11}\xi_1^2}{\lambda_1}\geq -\frac{2}{F}\left(\sum_iF^{ii}\xi_i\right)^2+\frac{F^{11}\xi_1^2}{\lambda_1}.
\end{align*}
\end{lemm}

\begin{proof}
We can write
\begin{align*}
F=\frac{\sigma_n}{\sigma_{n-1}}=\frac{1}{\sum_i \frac{1}{\lambda_i}}.
\end{align*}

Thus
\begin{align*}
F^{ii}=\frac{1}{\left(\sum_i \frac{1}{\lambda_i}\right)^2}\cdot\frac{1}{\lambda_i^2}=\frac{F^2}{\lambda_i^2}.
\end{align*}

It follows that 
\begin{align*}
F^{ii,jj}=\frac{2FF^{jj}}{\lambda_i^2}-\frac{2F^2\delta_{ij}}{\lambda_i^3}=\frac{ 2F^{ii}F^{jj}}{F}-\frac{2F^{ii}\delta_{ij}}{\lambda_i}.
\end{align*}

Consequently,
\begin{align*}
-\sum_{i,j}F^{ii,jj}\xi_i\xi_j=&\ -\sum_{i,j}\frac{ 2F^{ii}F^{jj}\xi_i\xi_j}{F}+\sum_i\frac{2F^{ii}\xi_i^2}{\lambda_i}.\\
\geq &\ -\frac{2}{F}\left(\sum_iF^{ii}\xi_i\right)^2+\frac{2F^{11}\xi_1^2}{\lambda_1}.
\end{align*}

Therefore,
\begin{align*}
-\sum_{i,j}F^{ii,jj}\xi_i\xi_j-\frac{F^{11}\xi_1^2}{\lambda_1}\geq -\frac{2}{F}\left(\sum_iF^{ii}\xi_i\right)^2+\frac{F^{11}\xi_1^2}{\lambda_1}.
\end{align*}

\end{proof}

For $k=n-2$, the computation is slightly more involved. Recall that when $n=3$, the following lemma is essentially contained in our previous work \cite{Lu23} using a direct but complicated computation. The current method is much simpler and works for all $n\geq 3$. 
\begin{lemm}\label{Con-lemm-2}[Guan and Sroka \cite{GS}]
Let $n\geq 3$ and let $F=\frac{\sigma_n}{\sigma_{n-2}}$ with $\lambda_1\geq\cdots\geq\lambda_n>0$. Then for any $\xi\in \mathbb{R}^n$, we have
\begin{align*}
-\sum_{i,j}F^{ii,jj}\xi_i\xi_j-\frac{F^{11}\xi_1^2}{\lambda_1}\geq  -\frac{2}{F}\left(\sum_iF^{ii}\xi_i\right)^2+\frac{1}{2(n-1)}\frac{F^{11}\xi_1^2}{\lambda_1}.
\end{align*}
\end{lemm}

\begin{proof}
We can write
\begin{align*}
F=\frac{\sigma_n}{\sigma_{n-2}}=\frac{1}{\sum_{i\neq j} \frac{1}{\lambda_i\lambda_j}}.
\end{align*}

Thus
\begin{align}\label{Concav-1}
F^{ii}=\frac{1}{\left(\sum_{i\neq j}\frac{1}{\lambda_i\lambda_j}\right)^2}\cdot\frac{1}{\lambda_i^2} \left(\sum_{k\neq i} \frac{1}{\lambda_k}\right)=\frac{F^2}{\lambda_i^2} \left(\sum_{k\neq i} \frac{1}{\lambda_k}\right).
\end{align}

It follows that for $i\neq j$, we have
\begin{align*}
F^{ii,jj}= \frac{2FF^{jj}}{\lambda_i^2}\left(\sum_{k\neq i} \frac{1}{\lambda_k}\right)-\frac{F^2}{\lambda_i^2\lambda_j^2}= \frac{2F^{ii}F^{jj}}{F}-\frac{F^2}{\lambda_i^2\lambda_j^2}.
\end{align*}

Similarly, 
\begin{align*}
F^{ii,ii}=\frac{2FF^{ii}}{\lambda_i^2} \left(\sum_{k\neq i} \frac{1}{\lambda_k}\right)-\frac{2F^2}{\lambda_i^3} \left(\sum_{k\neq i} \frac{1}{\lambda_k}\right)=\frac{2(F^{ii})^2}{F}-\frac{2F^{ii}}{\lambda_i}.
\end{align*}

Consequently,
\begin{align*}
-\sum_{i,j}F^{ii,jj}\xi_i\xi_j=&\ -\sum_{i,j} \frac{2F^{ii}F^{jj}\xi_i\xi_j}{F}+\sum_{i\neq j}\frac{F^2\xi_i\xi_j}{\lambda_i^2\lambda_j^2}+\sum_i \frac{2F^{ii}\xi_i^2}{\lambda_i}\\
= &\ -\frac{2}{F}\left(\sum_iF^{ii}\xi_i\right)^2+\sum_{i\neq j}\frac{F^2\xi_i\xi_j}{\lambda_i^2\lambda_j^2}+\sum_i \frac{2F^{ii}\xi_i^2}{\lambda_i}.
\end{align*}

For the sake of simplicity, define
\begin{align*}
\eta_i=\frac{\xi_i}{\lambda_i^2}.
\end{align*}

Together with (\ref{Concav-1}), we have
\begin{align}\label{Concav-2}
&\ -\sum_{i,j}F^{ii,jj}\xi_i\xi_j-\frac{F^{11}\xi_1^2}{\lambda_1}\\\nonumber
=&\ -\frac{2}{F}\left(\sum_iF^{ii}\xi_i\right)^2+\sum_{i\neq j}\frac{F^2\xi_i\xi_j}{\lambda_i^2\lambda_j^2}+\frac{F^{11}\xi_1^2}{\lambda_1}+\sum_{i\neq 1} \frac{2F^{ii}\xi_i^2}{\lambda_i}\\\nonumber
=&\ -\frac{2}{F}\left(\sum_iF^{ii}\xi_i\right)^2+\sum_{i\neq j}F^2\eta_i\eta_j+F^2\lambda_1\left(\sum_{k\neq 1}\frac{1}{\lambda_k}\right)\eta_1^2+\sum_{i\neq 1}2 F^2\lambda_i\left(\sum_{k\neq i}\frac{1}{\lambda_k}\right)\eta_i^2\\\nonumber
=&\ -\frac{2}{F}\left(\sum_iF^{ii}\xi_i\right)^2+F^2\left(\sum_{i\neq j}\eta_i\eta_j+\lambda_1\left(\sum_{k\neq 1}\frac{1}{\lambda_k}\right)\eta_1^2+\sum_{i\neq 1}2\lambda_i\left(\sum_{k\neq i}\frac{1}{\lambda_k}\right)\eta_i^2\right)\\\nonumber
=&\ -\frac{2}{F}\left(\sum_iF^{ii}\xi_i\right)^2+F^2\cdot I.
\end{align}

We have
\begin{align}\label{Concav-3}
I\geq &\ \sum_{i,j\neq n;i\neq j}\eta_i\eta_j+2\sum_{i\neq n}\eta_i\eta_n+\lambda_1\left(\sum_{k\neq 1}\frac{1}{\lambda_k}\right)\eta_1^2+\sum_{i\neq 1,n}\frac{2\lambda_i\eta_i^2}{\lambda_n} \\\nonumber
&\ +2\lambda_n\left(\sum_{k\neq n}\frac{1}{\lambda_k}\right)\eta_n^2.
\end{align}

For each $i\neq n$, we have
\begin{align*}
2\eta_i\eta_n+\frac{2\lambda_n\eta_n^2}{\lambda_i}\geq -\frac{\lambda_i\eta_i^2}{2\lambda_n}.
\end{align*}

Plugging into (\ref{Concav-3}), together with (\ref{Concav-1}) and the fact that $n\geq 3$, we have
\begin{align*}
I\geq &\ \sum_{i,j\neq n;i\neq j}\eta_i\eta_j+\lambda_1\left(\sum_{k\neq 1,n}\frac{1}{\lambda_k}+\frac{1}{2\lambda_n}\right)\eta_1^2+\sum_{i\neq 1,n}\frac{3\lambda_i\eta_i^2}{2\lambda_n}\\
\geq &\ \left(\sum_{i\neq n}\eta_i\right)^2-\sum_{i\neq n}\eta_i^2+(n-2)\eta_1^2+\frac{\lambda_1\eta_1^2}{2\lambda_n}+\sum_{i\neq 1,n}\frac{3\eta_i^2}{2}\\
\geq &\ \frac{\lambda_1\eta_1^2}{2\lambda_n}\geq  \frac{\lambda_1}{2(n-1)}\left(\sum_{k\neq 1}\frac{1}{\lambda_k}\right)\eta_1^2=\frac{1}{2(n-1)}\frac{F^{11}\xi_1^2}{\lambda_1F^2}.
\end{align*}

Plugging into (\ref{Concav-2}), we have
\begin{align*}
-\sum_{i,j}F^{ii,jj}\xi_i\xi_j-\frac{F^{11}\xi_1^2}{\lambda_1}\geq -\frac{2}{F}\left(\sum_iF^{ii}\xi_i\right)^2+\frac{1}{2(n-1)}\frac{F^{11}\xi_1^2}{\lambda_1}.
\end{align*}

\end{proof}

\section{Jabobi inequality}

In this section, we will derive a Jacobi inequality using the crucial concavity inequality. Due to the robust concavity inequality, we do not need to assume $\lambda_1$ is sufficiently large and our argument is much simpler compared to the proof for $n=3$ in our previous work \cite{Lu23}.

\begin{lemm}\label{Jacobi}
Let $k=n-1,n-2$ and let $n\geq 3$.  Let $f\in C^{1,1}(B_{10}\times \mathbb{R})$ be a positive function and let $u\in C^4(B_{10})$ be a convex solution of (\ref{eq-orginal}). For any $x_0\in B_9$, suppose that $D^2u$ is diagonalized at $x_0$ such that $\lambda_1\geq \cdots\geq \lambda_n$. Let $b=\ln\lambda_1$, then at $x_0$, we have
\begin{align*}
\sum_i F^{ii}b_{ii}\geq c(n)\sum_i F^{ii}b_i^2-C,
\end{align*}
in the viscosity sense, where $c(n)$ is a positive constant depending only on $n$ and $C$ depends only on $n$, $\|u\|_{C^{0,1}(\overline{B}_9)}$, $\min_{\overline{B}_9\times [-M,M]}f $ and $\|f\|_{C^{1,1}\left( \overline{B}_9\times [-M,M]\right)}$. Here $M$ is a large constant satisfying $\|u\|_{L^\infty(\overline{B}_9)}\leq M$.
\end{lemm}

\begin{proof}

Suppose $\lambda_1$ has multiplicity $m$ at $x_0$. By Lemma 5 in \cite{BCD}, we have
\begin{align}\label{BCD}
\delta_{\alpha\beta}\cdot{\lambda_1}_i=u_{\alpha\beta i},\quad 1\leq \alpha,\beta\leq m,
\end{align}
\begin{align*}
{\lambda_1}_{ii}\geq u_{11ii}+2\sum_{p>m}\frac{u_{1pi}^2}{\lambda_1 -\lambda_p},
\end{align*}
in the viscosity sense.

It follows that
\begin{align*}
b_{ii}=\frac{{\lambda_1}_{ii}}{\lambda_1}-\frac{{\lambda_1}_i^2}{\lambda_1^2}\geq \frac{u_{11ii}}{ \lambda_1}+2\sum_{p>m}\frac{u_{1pi}^2}{\lambda_1 (\lambda_1 -\lambda_p )}-\frac{ u_{11i}^2}{\lambda_{1}^2},
\end{align*}
in the viscosity sense.

Contracting with $F^{ii}$, we have
\begin{align}\label{m-0}
\sum_i F^{ii}b_{ii}\geq \sum_i \frac{F^{ii}u_{11ii}}{ \lambda_1}+2\sum_i\sum_{p>m}\frac{F^{ii}u_{1pi}^2}{\lambda_1 (\lambda_1 -\lambda_p )}-\sum_i \frac{F^{ii} u_{11i}^2}{\lambda_{1}^2}.
\end{align}

Differentiating (\ref{eq-orginal}), we have
\begin{align}\label{m-0-1}
\sum_i F^{ii}u_{ii11}+\sum_{p,q,r,s} F^{pq,rs}u_{pq1}u_{rs1}=f_{11}\geq -C\lambda_1-C,
\end{align}
where $C$ depends only on $\|u\|_{C^{0,1}(\overline{B}_9)}$ and $\|f\|_{C^{1,1}\left( \overline{B}_9\times [-M,M]\right)}$. Here $M$ is a large constant satisfying $\|u\|_{L^\infty(\overline{B}_9)}\leq M$. 

In the following, we will denote $C$ to be a constant depending only on $n$, $\|u\|_{C^{0,1}(\overline{B}_9)}$, $\min_{\overline{B}_9\times [-M,M]}f $ and $\|f\|_{C^{1,1}\left( \overline{B}_9\times [-M,M]\right)}$. It may change from line to line. 

Plugging (\ref{m-0-1}) into (\ref{m-0}), we have
\begin{align}\label{m-1}
\sum_i F^{ii}b_{ii}\geq -\sum_{p,q,r,s} \frac{F^{pq,rs}u_{pq1}u_{rs1}}{\lambda_1 }+2\sum_i\sum_{p>m}\frac{F^{ii}u_{1pi}^2}{\lambda_1 (\lambda_1 -\lambda_p )}-\sum_i \frac{ F^{ii}u_{11i}^2}{\lambda_{1}^2}-C.
\end{align}

By Lemma \ref{F^{ijkl}}, we have
\begin{align*}
-\sum_{p,q,r,s} F^{pq,rs}u_{pq1}u_{rs1}=-\sum_{i,j}F^{ii,jj}u_{ii1}u_{jj1}-\sum_{i\neq j}F^{ij,ji}u_{ij1}^2.
\end{align*}

Plugging into (\ref{m-1}), we have
\begin{align}\label{m-2}
\sum_i F^{ii}b_{ii}\geq &\ -\sum_{i, j} \frac{F^{ii,jj}u_{ii1}u_{jj1}}{\lambda_1} -\sum_{i\neq j} \frac{F^{ij,ji}u_{ij1}^2}{\lambda_1}+2\sum_i\sum_{p>m}\frac{F^{ii}u_{1pi}^2}{\lambda_1 (\lambda_1 -\lambda_p )}\\\nonumber
&\ -\sum_i \frac{ F^{ii}u_{11i}^2}{\lambda_{1}^2}-C.
\end{align}

By Lemma \ref{F^{ijkl}}, we have
\begin{align*}
-\sum_{i\neq j} \frac{F^{ij,ji}u_{ij1}^2}{\lambda_1}\geq -2\sum_{i>m} \frac{F^{i1,1i} u_{i11}^2}{\lambda_1}=2\sum_{i>m} \frac{(F^{ii}-F^{11})u_{11i}^2}{\lambda_1(\lambda_1-\lambda_i)}.
\end{align*}

On the other hand,
\begin{align*}
2\sum_i\sum_{p>m}\frac{F^{ii}u_{1pi}^2}{\lambda_1 (\lambda_1 -\lambda_p )} \geq  2\sum_{p>m}\frac{F^{11}u_{1p1}^2}{\lambda_1 (\lambda_1 -\lambda_p )}=2\sum_{i>m}\frac{F^{11}u_{11i}^2}{\lambda_1 (\lambda_1 -\lambda_i )}.
\end{align*}

Combining the above two inequalities and plugging into (\ref{m-2}), we have
\begin{align}\label{m-3}
\sum_i F^{ii}b_{ii}\geq &\ -\sum_{i,j} \frac{F^{ii,jj} u_{ii1}u_{jj1}}{\lambda_1}+2\sum_{i>m}\frac{F^{ii}u_{11i}^2}{\lambda_1 (\lambda_1 -\lambda_i )}-\sum_i \frac{ F^{ii}u_{11i}^2}{\lambda_{1}^2}-C.
\end{align}

By (\ref{BCD}), we have
\begin{align}\label{BCD2}
u_{11i}=u_{1i1}=\delta_{1i}\cdot(\lambda_1)_1=0,\quad 1<i\leq m.
\end{align}

Plugging into (\ref{m-3}), we have
\begin{align*}
\sum_i F^{ii}b_{ii}\geq  -\sum_{i,j} \frac{F^{ii,jj} u_{ii1}u_{jj1}}{\lambda_1}+\sum_{i>m}\frac{F^{ii}u_{11i}^2}{\lambda_1^2}- \frac{ F^{11}u_{111}^2}{\lambda_{1}^2}-C.
\end{align*}

Together with Lemma \ref{Con-lemm-1}, Lemma \ref{Con-lemm-2} and (\ref{BCD2}), we have
\begin{align*}
\sum_i F^{ii}b_{ii}\geq & -\frac{2}{\lambda_1F}\left(\sum_iF^{ii}u_{ii1}\right)^2+\sum_{i>m}\frac{F^{ii}u_{11i}^2}{\lambda_1^2}+c(n)\frac{ F^{11}u_{111}^2}{\lambda_{1}^2}-C\\
\geq &\ -\frac{2f_1^2}{\lambda_1F}+c(n)\sum_i \frac{F^{ii}u_{11i}^2}{\lambda_1^2}-C\\
\geq &\ c(n)\sum_iF^{ii}b_i^2-C.
\end{align*}

The lemma is now proved.

\end{proof}

\section{Legendre transform}

In this section, we adopt the idea by Shankar and Yuan \cite{SY-20} to use Legendre transform to obtain the bound of $b$ via its integral. The key observation is that after the transformation, $b$ is a subsolution of a uniformly elliptic equation. 

\medskip

For each $K>0$, define the Legendre transform for the function $u+\frac{K}{2}|x|^2$. We have
\begin{align*}
(x,Du(x)+Kx)=(Dw(y),y),
\end{align*}
where $w(y)$ is the Legendre transform for the function $u+\frac{K}{2}|x|^2$. Note that $y(x)=Du(x)+Kx$ is a diffeomorphism.

Define 
\begin{align*}
G(D^2w(y))=-F\left(-KI+\left(D^2w(y)\right)^{-1}\right)=-F(D^2u(x)).
\end{align*}

Suppose $D^2u$ is diagonalized at $x_0$. Then at $x_0$, we have
\begin{align}\label{g}
G^{ii}=F^{ii}\cdot w_{ii}^{-2}=F^{ii}(K+u_{ii})^2.
\end{align} 

\begin{lemm}\label{Jacobi-g}
Let $k=n-1,n-2$ and let $n\geq 3$. Let $f\in C^{1,1}(B_{10}\times \mathbb{R})$ be a positive function and let $u\in C^4(B_{10})$ be a convex solution of (\ref{eq-orginal}). For any $x_0\in B_9$, suppose that $D^2u$ is diagonalized at $x_0$ such that $\lambda_1\geq \cdots\geq \lambda_n$. Let $b=\ln\lambda_1$, then at $y(x_0)$, we have
\begin{align*}
\sum_i G^{ii}b^*_{ii}\geq -C,
\end{align*}
in the viscosity sense, where $C$ depends only on  $n$, $K$, $\|u\|_{C^{0,1}(\overline{B}_9)}$, $\min_{\overline{B}_9\times [-M,M]}f $ and $\|f\|_{C^{1,1}\left( \overline{B}_9\times [-M,M]\right) }$. Here $M$ is a large constant satisfying $\|u\|_{L^{\infty}(\overline{B}_9)}\leq M$. We use the notation $b^*(y)=b(y(x))$.
\end{lemm}

\begin{proof}
By chain rule, we have
\begin{align}\label{chain}
\frac{\partial b}{\partial x_i}=\sum_k\frac{\partial y^k}{\partial x^i}\frac{\partial b^*}{\partial y^k}=\sum_k(K\delta_{ik}+u_{ik})b^*_k.
\end{align}

Consequently,
\begin{align*}
\frac{\partial^2 b}{\partial x_i^2}=&\ \sum_k u_{iik}b^*_k+\sum_{k,l}(K\delta_{ik}+u_{ik})b^*_{kl}(K\delta_{il}+u_{il})\\
=&\ \sum_k u_{iik}b^*_k+(K+u_{ii})^2b^*_{ii}.
\end{align*}

Together with (\ref{g}), we have
\begin{align*}
\sum_iF^{ii}b_{ii}=&\ \sum_{i,k}F^{ii}u_{iik}b^*_k+\sum_{i}F^{ii}(K+u_{ii})^2b^*_{ii}\\
=&\ \sum_kf_kb^*_k+\sum_iG^{ii}b^*_{ii}.
\end{align*}

Together with Lemma \ref{Jacobi}, we have
\begin{align*}
\sum_iG^{ii}b^*_{ii}=\sum_iF^{ii}b_{ii}-\sum_kf_kb^*_k\geq c(n)\sum_iF^{ii}b_i^2-\sum_kf_kb^*_k-C,
\end{align*}
where $C$ is a constant depending only on $n$, $K$, $\|u\|_{C^{0,1}(\overline{B}_9)}$, $\min_{\overline{B}_9\times [-M,M]}f $ and $\|f\|_{C^{1,1}\left( \overline{B}_9\times [-M,M]\right)}$. It may change from line to line. 

By (\ref{g}) and (\ref{chain}), we have
\begin{align*}
\sum_iF^{ii}b_i^2=\sum_i G^{ii}(K+u_{ii})^{-2}\bigg( (K+u_{ii})b_i^*\bigg)^2=\sum_iG^{ii}(b^*_i)^2.
\end{align*}

It follows that
\begin{align}\label{g-1}
\sum_iG^{ii}b^*_{ii}\geq&\ c(n)\sum_iG^{ii}(b^*_i)^2-\sum_kf_kb^*_k-C\\\nonumber
\geq&\  c(n)\sum_iG^{ii}(b^*_i)^2-C\sum_k|b^*_k|-C\\\nonumber
\geq&\ -C\sum_i \frac{1}{G^{ii}}-C.
\end{align}

By Lemma \ref{Sigma_k-Lemma-0}, Lemma \ref{F^{ijkl}} and (\ref{g}), we have
\begin{align}\label{g-3}
G^{ii}=&\  \left( \frac{\sigma_n^{ii}}{\sigma_k}-\frac{\sigma_n\sigma_k^{ii}}{\sigma_k^2} \right)\left( K+u_{ii} \right)^2\\\nonumber
=&\ \frac{\sigma_n}{\sigma_k}\left(\frac{1}{\lambda_i}-\frac{\sigma_k^{ii}}{\sigma_k}\right)\left(K+\lambda_i\right)^2\\\nonumber
=&\ f\cdot \frac{\sigma_k-\lambda_i\sigma_k^{ii}}{\lambda_i\sigma_k}\left(K+\lambda_i\right)^2\\\nonumber
=&\ f\cdot\frac{\sigma_k(\lambda|i)}{\lambda_i\sigma_k}\left(K+\lambda_i\right)^2.
\end{align}

{\bf Case 1}: $k=n-1$. 

In this case, by (\ref{g-3}), we have
\begin{align*}
G^{ii}=f \cdot \frac{\sigma_n}{\lambda_i^2\sigma_{n-1}}\left(K+\lambda_i\right)^2= \frac{f^2\left(K+\lambda_i\right)^2}{\lambda_i^2}.
\end{align*}

Together with Lemma \ref{lambda_3-1}, we have
\begin{align}\label{g-2-1}
\frac{1}{C_0}\leq G^{ii}\leq C_0,\quad \forall 1\leq i\leq n.
\end{align}
where $C_0$ is a positive constant depending only on $n$, $K$, $\min_{\overline{B}_9\times [-M,M]}f $ and $\max_{\overline{B}_9\times [-M,M]}f $. 

Plugging into (\ref{g-1}), we conclude that
\begin{align*}
\sum_iG^{ii}b^*_{ii}\geq -C\sum_i \frac{1}{G^{ii}}-C\geq -C.
\end{align*}

{\bf Case 2}: $k=n-2$.

In this case, for $1\leq i\leq n-1$, by (\ref{g-3}) and Lemma \ref{lambda_3}, we have
\begin{align*}
G^{ii}\leq &\ C_0\frac{\sigma_{n-1}}{\lambda_i^2\sigma_{n-2}}\cdot\lambda_i^2\leq  C_0\lambda_{n-1},\\
G^{ii}\geq &\ \frac{1}{C_0}\frac{\sigma_{n-1}}{\lambda_i^2\sigma_{n-2}}\cdot\lambda_i^2\geq \frac{\lambda_{n-1}}{C_0}.
\end{align*}

For $i=n$,  by (\ref{g-3}), Lemma \ref{lambda_3} and the fact that $\lambda_n=\frac{f\sigma_{n-2}}{\lambda_1\cdots\lambda_{n-1}}$, we have
\begin{align*}
G^{nn}\leq  &\ \frac{C_0}{\lambda_n}= \frac{C_0\lambda_1\cdots\lambda_{n-1}}{f\sigma_{n-2}}\leq  C_0\lambda_{n-1},\\
G^{nn}\geq &\  \frac{1}{C_0\lambda_n}= \frac{\lambda_1\cdots\lambda_{n-1}}{C_0f\sigma_{n-2}}\geq \frac{\lambda_{n-1}}{C_0}.
\end{align*}

Therefore,
\begin{align}\label{g-2}
\frac{\lambda_{n-1}}{C_0}\leq  G^{ii}\leq C_0\lambda_{n-1},\quad \forall 1\leq i\leq n.
\end{align}
where $C_0$ is a positive constant depending only on $n$, $K$, $\min_{\overline{B}_9\times [-M,M]}f $ and $\max_{\overline{B}_9\times [-M,M]}f $. 

Together with Lemma \ref{lambda_3}, we have
\begin{align*}
-\sum_i\frac{1}{G^{ii}}\geq -\frac{C}{\lambda_{n-1}}\geq -C.
\end{align*}

Plugging into (\ref{g-1}), we conclude that
\begin{align*}
\sum_iG^{ii}b^*_{ii}\geq -C\sum_i \frac{1}{G^{ii}}-C\geq -C.
\end{align*}

The lemma is now proved.
\end{proof}

We now state the mean value inequality. 
\begin{lemm}\label{mean}
Let $k=n-1,n-2$ and let $n\geq 3$. Let $f\in C^{1,1}(B_{10}\times \mathbb{R})$ be a positive function and let $u\in C^4(B_{10})$ be a convex solution of (\ref{eq-orginal}). Let $b=\ln\lambda_1$, where $\lambda_1$ is the largest eigenvalue of $D^2u$. Then we have
\begin{align*}
b(0)\leq  C\int_{B_1}b(x)\sigma_{n-1}(D^2u(x)) dx+C,
\end{align*}
where $C$ depends only on $n$, $\|u\|_{C^{0,1}(\overline{B}_9)}$, $\min_{\overline{B}_9\times [-M,M]}f $ and $\|f\|_{C^{1,1}\left( \overline{B}_9\times [-M,M]\right) }$. Here $M$ is a large constant satisfying $\|u\|_{L^\infty(\overline{B}_9)}\leq M$. 
\end{lemm}

\begin{proof}
Let $K=1$ in Lemma \ref{Jacobi-g}. Without loss of generality, we may assume $Du(0)=0$, i.e. $y(0)=0$.  

By Lemma \ref{lambda_3-1}, Lemma \ref{lambda_3}, (\ref{g-2-1}) and (\ref{g-2}), we have
\begin{align*}
\sum_i G^{ii}\geq C_0,
\end{align*}
where $C_0$ is a positive constant depending only on $n$, $\min_{\overline{B}_9\times [-M,M]}f $ and $\max_{\overline{B}_9\times [-M,M]}f $.

Consequently,
\begin{align*}
\sum_i G^{ii}\left( b^*(y)+A|y|^2\right)_{ii}\geq 0,
\end{align*}
where $A$ is a positive constant depending only on $n$, $\min_{\overline{B}_9\times [-M,M]}f $ and $\max_{\overline{B}_9\times [-M,M]}f $. 

By (\ref{g-2-1}), $G^{ii}$ is uniformly elliptic for $k=n-1$. By (\ref{g-2}), $\frac{G^{ii}}{\lambda_{n-1}} $ is uniformly elliptic for $k=n-2$. By local maximum principle (Theorem 4.8 in \cite{CC}), we have
\begin{align*}
b^*(0)+A|0|^2\leq C\int_{B_1^y}\left( b^*(y)+A|y|^2\right) dy.
\end{align*}

Since $D^2u>0$, $y(x)=Du(x)+x$ is uniformly monotone, i.e. $|y(x_I)-y(x_{II})|\geq |x_I-x_{II}|$, we have $x(B_1^y)\subset B_1$. Together with the fact that $y(0)=0$ and $\sigma_n=f\sigma_k$, we have
\begin{align*}
b(0)= b^*(0)\leq &\ C\int_{B_1}b(x)\det(D^2u(x)+I)dx+C\\
\leq &\ C\int_{B_1}b\cdot(\lambda_1+1)\cdots(\lambda_n+1) dx+C\\
\leq &\ C\int_{B_1}b\left(1+\sigma_1+\cdots+\sigma_n\right)  dx+C\\
\leq &\ C\int_{B_1}b\left(1+\sigma_1+\cdots+\sigma_{n-1}\right)  dx+C.
\end{align*}

By Lemma \ref{lambda_3-1} and Lemma \ref{lambda_3}, we have
\begin{align*}
\sigma_{n-1}\geq \lambda_1\cdots\lambda_{n-1}\geq C\lambda_1\cdots\lambda_{n-2}\geq C\sigma_{n-2}.
\end{align*}

Similarly, we can obtain
\begin{align*}
\sigma_{n-2}\geq C\sigma_{n-3}\geq \cdots\geq C\sigma_1\geq C.
\end{align*}

It follows that
\begin{align*}
b(0)\leq C\int_{B_1}b\sigma_{n-1} dx+C.
\end{align*}

The lemma is now proved.
\end{proof}

\section{Proof of Theorem \ref{Theorem} and Theorem \ref{n=2}} 

In this section, we complete the proof of Theorem \ref{Theorem} and Theorem \ref{n=2} via integration by parts. Since $F^{ij}$ is not divergence free, we will choose $H^{ij}=\sigma_kF^{ij}$ to perform the integration by parts as in \cite{Lu23}.

Define
\begin{align}\label{H1}
H^{ij}=\sigma_kF^{ij}=\sigma_n^{ij}-\frac{\sigma_n\sigma_k^{ij}}{\sigma_k}.
\end{align}

Before we start to prove the main theorem, we first prove an induction lemma. The lemma will help us to simplify the argument in the integration by parts.
\begin{lemm}\label{Cauchy}
Let $k=n-1,n-2$ and let $n\geq 3$. Let $f\in C^{1,1}(B_{10}\times \mathbb{R})$ be a positive function and let $u\in C^4(B_{10})$ be a convex solution of (\ref{eq-orginal}). For any $x_0\in B_9$, suppose that $D^2u$ is diagonalized at $x_0$. Let $b=\ln\lambda_1$, where $\lambda_1$ is the largest eigenvalue of $D^2u$. Then at $x_0$, we have
\begin{align*}
\sum_i |b_i\sigma_l^{ii}|\leq C_0\epsilon\sum_i H^{ii}b_i^2+\frac{C_0}{\epsilon}\sigma_k,\quad \forall 1\leq l\leq k,
\end{align*}
where $\epsilon$ is any positive constant and $C_0$ depends only on $n$, $\min_{\overline{B}_9\times [-M,M]}f $ and $\max_{\overline{B}_9\times [-M,M]}f $. Here $M$ is a large constant satisfying $\|u\|_{L^\infty(\overline{B}_9)}\leq M$.
\end{lemm}

\begin{proof}

Without loss of generality, we may assume $\lambda_1\geq \cdots\geq \lambda_n$.

By Lemma \ref{Sigma_k-Lemma-0} and Lemma \ref{F^{ijkl}}, we have
\begin{align}\label{H}
H^{ii}=&\  \sigma_k\left( \frac{\sigma_n^{ii}}{\sigma_k}-\frac{\sigma_n\sigma_k^{ii}}{\sigma_k^2} \right)=\frac{\sigma_n}{\sigma_k}\left(\frac{\sigma_k}{\lambda_i}-\sigma_k^{ii}\right)\\\nonumber
=&\ f\cdot \frac{\sigma_k-\lambda_i\sigma_k^{ii}}{\lambda_i}=f\cdot\frac{\sigma_k(\lambda|i)}{\lambda_i}.
\end{align}

{\bf Case 1}: $k=n-1$.

In this case, by (\ref{H}), we have
\begin{align*}
H^{ii}=f\cdot\frac{\sigma_n}{\lambda_i^2}.
\end{align*}

For $1\leq i\leq l$, by Lemma \ref{lambda_3-1} and the fact that $1\leq l\leq n-1$, we have
\begin{align*}
|b_i\sigma_l^{ii}|\leq C_0|b_i|\frac{\sigma_l}{\lambda_i}\leq C_0\epsilon \frac{\sigma_nb_i^2}{\lambda_i^2}+\frac{C_0}{\epsilon}\frac{\sigma_l^2}{\sigma_n}\leq C_0\epsilon H^{ii}b_i^2+\frac{C_0}{\epsilon} \sigma_l.
\end{align*}

For $l<i\leq n$, by Lemma \ref{lambda_3-1} and the fact that $1\leq l\leq n-1$, we have
\begin{align*}
|b_i\sigma_l^{ii}|\leq C_0|b_i|\sigma_{l-1}\leq C_0\epsilon \frac{\sigma_nb_i^2}{\lambda_i^2}+\frac{C_0}{\epsilon}\frac{\lambda_i^2\sigma_{l-1}^2}{\sigma_n}\leq C_0\epsilon H^{ii}b_i^2+\frac{C_0}{\epsilon}\sigma_l.
\end{align*}

Therefore,
\begin{align*}
\sum_i |b_i\sigma_l^{ii}|\leq C_0\epsilon\sum_i H^{ii}b_i^2+\frac{C_0}{\epsilon}\sigma_{n-1},\quad \forall 1\leq l\leq n-1.
\end{align*}

{\bf Case 2}: $k=n-2$.

In this case, for $1\leq i\leq n-1$, by (\ref{H}), we have
\begin{align*}
H^{ii}=f\cdot\frac{\sigma_{n-2}(\lambda|i)}{\lambda_i}\geq C_0\frac{\sigma_{n-1}}{\lambda_i^2}.
\end{align*}

For $i=n$, by (\ref{H}), we have
\begin{align*}
H^{nn}=f\cdot\frac{\sigma_{n-2}(\lambda|n)}{\lambda_n}\geq C_0\frac{\sigma_{n-2}}{\lambda_n}.
\end{align*}

For $1\leq i\leq l$, by Lemma \ref{lambda_3} and the fact that $1\leq l\leq n-2$, we have
\begin{align*}
|b_i\sigma_l^{ii}|\leq C_0|b_i|\frac{\sigma_l}{\lambda_i}\leq C_0\epsilon \frac{\sigma_{n-1}b_i^2}{\lambda_i^2}+\frac{C_0}{\epsilon}\frac{\sigma_l^2}{\sigma_{n-1}}\leq C_0\epsilon H^{ii}b_i^2+\frac{C_0}{\epsilon} \sigma_l.
\end{align*}

For $l<i\leq n-1$, by Lemma \ref{lambda_3} and the fact that $1\leq l\leq n-2$, we have
\begin{align*}
|b_i\sigma_l^{ii}|\leq C_0|b_i|\sigma_{l-1}\leq C_0\epsilon \frac{\sigma_{n-1}b_i^2}{\lambda_i^2}+\frac{C_0}{\epsilon}\frac{\lambda_i^2\sigma_{l-1}^2}{\sigma_{n-1}}\leq C_0\epsilon H^{ii}b_i^2+\frac{C_0}{\epsilon}\sigma_l.
\end{align*}

For $i=n$, by Lemma \ref{lambda_3} and the fact that $1\leq l\leq n-2$, we have
\begin{align*}
|b_n\sigma_l^{nn}|\leq C_0|b_n|\sigma_{l-1}\leq C_0\epsilon \frac{\sigma_{n-2}b_n^2}{\lambda_n}+\frac{C_0}{\epsilon}\frac{\lambda_n\sigma_{l-1}^2}{\sigma_{n-2}}\leq C_0\epsilon H^{nn}b_n^2+\frac{C_0}{\epsilon}\sigma_{l-1}.
\end{align*}

Therefore,
\begin{align*}
\sum_i |b_i\sigma_l^{ii}|\leq C_0\epsilon\sum_i H^{ii}b_i^2+\frac{C_0}{\epsilon}\sigma_{n-2},\quad \forall 1\leq l\leq n-2.
\end{align*}

The lemma is now proved.
\end{proof}

Proof of Theorem \ref{Theorem} for $k=n-1$.
\begin{proof}
Let $b=\ln \lambda_1$. Let $\varphi$ be a cutoff function such that $\varphi=1$ on $B_1$ and $\varphi=0$ outside $B_2$. By Lemma \ref{Sigma_k-Lemma-0} and the fact that $\sum_i (\sigma_{n-1}^{ij})_i=0$, we have
\begin{align*}
\int_{B_1}b\sigma_{n-1}dx\leq&\ \int_{B_2}\varphi b \sigma_{n-1}dx=\frac{1}{n-1}\sum_{i,j}\int_{B_2}\varphi b \sigma_{n-1}^{ij}u_{ij}dx\\\nonumber
=&\ -\frac{1}{n-1}\sum_{i,j}\int_{B_2} (\varphi_i b+\varphi b_i) \sigma_{n-1}^{ij}u_j dx\\\nonumber
\leq &\ C\int_{B_2} b\sigma_{n-2} dx-\frac{1}{n-1}\sum_{i,j}\int_{B_2}\varphi b_i\sigma_{n-1}^{ij}u_j dx,
\end{align*}
where $C$ is a constant depending only on $n$ and $\|u\|_{C^{0,1}(\overline{B}_9)}$. 

In the following, we will denote $C$ to be a constant depending only on $n$, $\|u\|_{C^{0,1}(\overline{B}_9)}$, $\min_{\overline{B}_9\times [-M,M]}f $ and $\|f\|_{C^{1,1}\left( \overline{B}_9\times [-M,M]\right)}$. It may change from line to line. Here $M$ is a large constant satisfying $\|u\|_{L^\infty(\overline{B}_9)}\leq M$.

Together with Lemma \ref{mean}, we have
\begin{align}\label{p1-0}
b(0)\leq  C\int_{B_2} b\sigma_{n-2} dx-\frac{C}{n-1}\sum_{i,j}\int_{B_2}\varphi b_i\sigma_{n-1}^{ij}u_j dx+C.
\end{align}

By Lemma \ref{Cauchy}, we have
\begin{align*}
\left|-\frac{C}{n-1}\sum_{i,j}\int_{B_2}\varphi b_i\sigma_{n-1}^{ij}u_j dx\right|\leq C\sum_i\int_{B_2}H^{ii}b_i^2dx+C\int_{B_2}\sigma_{n-1}dx.
\end{align*}

It follows that 
\begin{align*}
b(0)\leq C\int_{B_2} b\sigma_{n-2} dx+C\sum_i\int_{B_2}H^{ii}b_i^2dx+C\int_{B_2}\sigma_{n-1}dx+C.
\end{align*}

With the help of Lemma \ref{Cauchy}, we can repeat the process and obtain
\begin{align}\label{p1-1}
b(0)\leq C\sum_i\int_{B_3}H^{ii}b_i^2dx+C\int_{B_3}\sigma_{n-1}dx+C.
\end{align}

By Lemma \ref{Jacobi}, we have
\begin{align*}
\sum_{i,j} H^{ij}b_{ij}\geq c(n)\sum_{i,j} H^{ij}b_ib_j -C\sigma_{n-1},
\end{align*}
in the viscosity sense.

By Theorem 1 in \cite{I}, $b$ also satisfies the above inequality in the distribution sense. Let $\Phi$ be a cutoff function such that $\Phi=1$ on $B_3$ and $\Phi=0$ outside $B_4$. Then
\begin{align*}
&\ \sum_{i,j}\int_{B_3}H^{ij}b_ib_jdx\leq\sum_{i,j} \int_{B_4}\Phi^2H^{ij}b_ib_jdx\\
\leq &\ \frac{1}{c(n)}\sum_{i,j} \int_{B_4}\Phi^2H^{ij}b_{ij} dx+C\int_{B_4}\sigma_{n-1}dx\\
\leq &\ -\frac{2}{c(n)} \sum_{i,j} \int_{B_4}\Phi\Phi_j H^{ij}b_i dx-C\sum_{i,j}\int_{B_4}\Phi^2 (H^{ij})_j b_idx+C\int_{B_4}\sigma_{n-1}dx\\
\leq &\ \frac{1}{2}\sum_{i}\int_{B_4} \Phi^2H^{ii}b_i^2dx+C\sum_i \int_{B_4}\Phi_i^2H^{ii}dx+C\sum_{i,j}\int_{B_4}\Phi^2 f_j\sigma_{n-1}^{ij} b_idx\\
&\ +C\int_{B_4}\sigma_{n-1}dx
\end{align*}
We have used the fact that $\sum_j (H^{ij})_j=\sum_j\left((\sigma_n^{ij})_j-(f\sigma_{n-1}^{ij})_j\right)=-\sum_j f_j\sigma_{n-1}^{ij}$ in the last inequality.

Consequently,
\begin{align*}
\frac{1}{2}\sum_{i,j}\int_{B_4} \Phi^2H^{ii}b_i^2dx\leq  C\sum_i \int_{B_4}\Phi_i^2H^{ii}dx+C\sum_i\int_{B_4}\Phi^2 f_i\sigma_{n-1}^{ii} b_idx+C\int_{B_4}\sigma_{n-1}dx.
\end{align*}

Together with Lemma \ref{Sigma_k-Lemma-0}, Lemma \ref{Cauchy} and the fact that $H^{ii}=\sigma_n^{ii}-f\sigma_{n-1}^{ii}$ , we have
\begin{align*}
&\ \frac{1}{2}\sum_{i,j}\int_{B_4} \Phi^2H^{ii}b_i^2dx\\
\leq &\ C\sum_i \int_{B_4} H^{ii}dx+C\epsilon\sum_i\int_{B_4}\Phi^2H^{ii}b_i^2  dx+\frac{C}{\epsilon}\int_{B_4}\sigma_{n-1}dx\\
\leq &\ C\int_{B_4}\left(C\sigma_{n-1}-C\sigma_{n-2}\right)dx+\frac{1}{4}\sum_{i}\int_{B_4} \Phi^2H^{ii}b_i^2dx+C\int_{B_4}\sigma_{n-1}dx\\
\leq &\ \frac{1}{4}\sum_{i,j}\int_{B_4} \Phi^2H^{ii}b_i^2dx+C\int_{B_4}\sigma_{n-1}dx.
\end{align*}

Consequently,
\begin{align*}
\frac{1}{4}\sum_{i,j}\int_{B_4} \Phi^2H^{ij}b_ib_jdx\leq C\int_{B_4}\sigma_{n-1}dx.
\end{align*}

Plugging into (\ref{p1-1}), we have
\begin{align*}
b(0)\leq C\int_{B_4}\sigma_{n-1}dx+C.
\end{align*}

Let $\Psi$ be a cutoff function such that $\Psi=1$ on $B_4$ and $\Psi=0$ outside $B_5$. Then
\begin{align*}
\int_{B_4}\sigma_{n-1}dx\leq \int_{B_5}\Psi\sigma_{n-1}dx=-\frac{1}{n-1}\sum_{i,j}\int_{B_5}\Psi_i\sigma_{n-1}^{ij}u_jdx\leq C\int_{B_5}\sigma_{n-2}dx.
\end{align*}

Thus
\begin{align*}
b(0)\leq C\int_{B_5}\sigma_{n-2}dx+C.
\end{align*}

Repeating the process, we have
\begin{align*}
b(0)\leq C\int_{B_6}\sigma_1dx+C\leq C.
\end{align*}

The theorem is now proved.
\end{proof}

Proof of Theorem \ref{Theorem} for $k=n-2$.
\begin{proof}
By (\ref{p1-0}), we have
\begin{align}\label{p-1}
b(0)\leq  C\int_{B_2} b\sigma_{n-2} dx-\frac{C}{n-1}\sum_{i,j}\int_{B_2}\varphi b_i\sigma_{n-1}^{ij}u_j dx+C.
\end{align}

For each $x_0\in B_2$, without loss of generality, we may assume $D^2u$ is diagonalized at $x_0$. By (\ref{H}), we have
\begin{align*}
H^{ii}=f\cdot\frac{\sigma_{n-2}(\lambda|i)}{\lambda_i}=\frac{f\sigma_{n-1}^{ii}}{\lambda_i}.
\end{align*}

Together with Lemma \ref{Sigma_k-Lemma-0} and the fact that $H^{ii}=\sigma_n^{ii}-f\sigma_{n-2}^{ii}$, we have
\begin{align*}
\sum_{i,j}|\varphi b_i\sigma_{n-1}^{ij}u_j|\leq &\ C\sum_i|b_i\sigma_{n-1}^{ii}|\leq  C \sum_i |b_i|\lambda_i H^{ii} \\
\leq&\  C\sum_{i}H^{ii}b_i^2+C\sum_i H^{ii}\lambda_i^2\\
=&\ C\sum_{i}H^{ii}b_i^2+C\bigg(\sigma_1\sigma_n-f\left(\sigma_1\sigma_{n-2}-(n-1)\sigma_{n-1}\right)\bigg)\\
\leq&\ C\sum_{i}H^{ii}b_i^2+C\sigma_{n-1}.
\end{align*}

Plugging into (\ref{p-1}), we have
\begin{align*}
b(0)\leq C\int_{B_2} b\sigma_{n-2} dx+C\sum_{i}\int_{B_2}H^{ii}b_i^2dx+C\int_{B_2}\sigma_{n-1}dx+C.
\end{align*}

The rest follows exactly the same as the proof of the case $k=n-1$. We omit the details.

\end{proof}

Proof of Theorem \ref{n=2}.
\begin{proof}
We first note that the concavity inequality (Lemma \ref{Con-lemm-1}) works for $n=2$ as well. Consequently, the Jacobi inequality (Lemma \ref{Jacobi}) works for $\frac{\sigma_2}{\sigma_1}$ as well. It follows that the mean value inequality (Lemma \ref{mean}) works for $\frac{\sigma_2}{\sigma_1}$. Now the estimate for the integral follows exactly the same argument as in the case $\frac{\sigma_n}{\sigma_{n-1}}$ for $n\geq 3$. The theorem is now proved.
\end{proof}

\section{Singular solution}

In this section, we will show that interior $C^2$ estimate fails for $\frac{\sigma_n}{\sigma_k}$ for $k\leq n-3$ by extending Pogorelov's example \cite{Pb}.

\medskip

Proof of Theorem \ref{Singular}.
\begin{proof}
Let
\begin{align*}
u^\sigma(x)=(1+x_1^2)\left(\sigma+\sum_{i=2}^nx_i^2\right)^{\frac{\alpha}{2}},
\end{align*}
where $\alpha=2-\frac{2}{n-k}$. Note that for $k=0$, it is exactly Pogorelov's example. Since $1\leq k\leq n-3$, we have $1<\alpha<2$.

It follows that
\begin{align*}
u^\sigma_{11}= 2\left(\sigma+\sum_{i=2}^nx_i^2\right)^{\frac{\alpha}{2}},\quad u^\sigma_{1i}= \left(\sigma+\sum_{i=2}^nx_i^2\right)^{\frac{\alpha}{2}-1 }2\alpha x_1x_i,\quad i\neq 1.
\end{align*}

Moreover,
\begin{align*}
u^\sigma_{ii}=&\ (1+x_1^2)\left(\sigma+\sum_{i=2}^nx_i^2\right)^{\frac{\alpha}{2}-2}\left(\alpha(\alpha-2)x_i^2+\alpha\left(\sigma+\sum_{i=2}^nx_i^2\right)\right),\quad i\neq 1,\\
u^\sigma_{ij}=&\ (1+x_1^2)\left(\sigma+\sum_{i=2}^nx_i^2\right)^{\frac{\alpha}{2}-2}\alpha(\alpha-2)x_ix_j,\quad i,j\neq 1,\quad i\neq j.
\end{align*}

It follows that there exists $r>0$ such that for all $x\in B_r$, we have
\begin{align*}
\sigma_n(D^2u^\sigma)=&\ \left(\sigma+\sum_{i=2}^nx_i^2\right)^{\frac{n\alpha}{2}-(n-1)}\varphi^\sigma,\\
\sigma_k(D^2u^\sigma)=&\ \left(\sigma+\sum_{i=2}^nx_i^2\right)^{\frac{k\alpha}{2}-k}\phi^\sigma,
\end{align*}
where $\varphi^\sigma$ and $\phi^\sigma$ are smooth positive function with positive lower bound in $B_r$.

Consequently,
\begin{align*}
\left(\frac{\sigma_n}{\sigma_k}\right) (D^2u^\sigma)=\left(\sigma+\sum_{i=2}^nx_i^2\right)^{\frac{(n-k)\alpha}{2}-(n-k-1)}f^\sigma.
\end{align*}
where $f^\sigma$ is smooth positive function with positive lower bound in $B_r$.

Recall that $\alpha=2-\frac{2}{n-k}$, thus
\begin{align*}
\left(\frac{\sigma_n}{\sigma_k}\right) (D^2u^\sigma)=f^\sigma.
\end{align*}

Let $\sigma\rightarrow0$, by \cite{Urbas}, $u^\sigma\rightarrow u$ and $f^\sigma\rightarrow f$ locally uniformly, where $u=(1+x_1^2)\left(\sum_{i=2}^nx_i^2\right)^{\frac{\alpha}{2}}$ and $f$ is smooth positive function with positive lower bound in $B_r$. Moreover, $u$ is a viscosity solution of
\begin{align*}
\left(\frac{\sigma_n}{\sigma_k}\right) (D^2u)=f.
\end{align*}

It is easy to see $u$ is convex and Lipschitz. On the other hand, by direct computation, we can see that $u\notin C^{1,\beta}\left( B_{\frac{r}{2}}\right) $ for any $\beta>1-\frac{2}{n-k}$. Consequently, $u$ is the desired singular solution.

\end{proof}

\begin{rema}
For general Hessian quotient equation $\left(\frac{\sigma_k}{\sigma_l}\right)(D^2u)=f(x)$ with $1\leq l<k\leq n$ and $k-l\geq 3$, we can construct the singular solution by the same method of Urbas \cite{Urbas}. First of all, take our singular solution for the Hessian quotient equation $\left(\frac{\sigma_k}{\sigma_l}\right)(D^2u)=f(x)$ in $\mathbb{R}^k$, then extend it trivially to $\mathbb{R}^n$. It follows from \cite{Urbas} that it is a singular solution for $\left(\frac{\sigma_k}{\sigma_l}\right)(D^2u)=f(x)$ in $\mathbb{R}^n$.
\end{rema}

\noindent

{\it Acknowledgement}: The author would like to thank Professors Pengfei Guan for enlightening conversations over concavity of Hessian quotient operator and for great generosity for sharing Lemma \ref{Con-lemm-1} and Lemma \ref{Con-lemm-2}. The author would also like to thank Xiangwen Zhang for helpful discussions.

\end{document}